\documentclass[oneside,english]{amsart}
\usepackage[T1]{fontenc}
\usepackage[latin9]{inputenc}
\usepackage{amsthm}
\usepackage{amstext}
\usepackage{esint}

\makeatletter
\numberwithin{equation}{section}
\numberwithin{figure}{section}
\theoremstyle{plain}
\newtheorem{thm}{\protect\theoremname}[section]
  \theoremstyle{plain}
  \newtheorem{lem}[thm]{\protect\lemmaname}

\def\makebbb#1{
    \expandafter\gdef\csname#1\endcsname{
        \ensuremath{\Bbb{#1}}}
}\makebbb{R}\makebbb{N}\makebbb{Z}\makebbb{C}\makebbb{H}\makebbb{E}\makebbb{H}\makebbb{P}\makebbb{B}\makebbb{Q}\makebbb{E}

\usepackage{babel}

\usepackage{babel}

\makeatother

\usepackage{babel}

\makeatother

\usepackage{babel}
  \providecommand{\lemmaname}{Lemma}
\providecommand{\theoremname}{Theorem}

\begin{document}

\title{On the optimal regularity of weak geodesics in the space of metrics
on a polarized manifold}

\author{Robert J. Berman}
\begin{abstract}
Let $(X,L)$ be a polarized compact manifold, i.e. $L$ is an ample
line bundle over $X$ and denote by $\mathcal{H}$ the infinite dimensional
space of all positively curved Hermitian metrics on $L$ equipped
with the Mabuchi metric. In this short note we show, using Bedford-Taylor
type envelope techniques developed in the authors previous work \cite{ber2},
that Chen's weak geodesic connecting any two elements in $\mathcal{H}$
are $C^{1,1}-$smooth, i.e. the real Hessian is bounded, for any fixed
time $t,$ thus improving the original bound on the Laplacins due
to Chen. This also gives a partial generalization of Blocki's refinement
of Chen's regularity result. More generally, a regularity result for
complex Monge-Ampère equations over $X\times D,$ for $D$ a pseudconvex
domain in $\C^{n}$ is given.
\end{abstract}
\maketitle

\section{Introduction}

Let $X$ be an $n-$dimensional compact complex manifold equipped
with a Kähler form $\omega$ and denote by $[\omega]$ the corresponding
cohomology class in $H^{2}(X,\R).$ The space of all Kähler metrics
in $[\omega]$ may be identified with the space $\mathcal{H}(X,\omega)$
of all Kähler potentials, modulo constants, i.e. the space of all
functions $u$ on $X$ such that 
\[
\omega_{u}:=\omega+dd^{c}u,\,\,\,\,\,\,(dd^{c}:=\frac{i}{2\pi}\partial\bar{\partial})
\]
 is positive, i.e. defines a Kähler form on $X.$ Mabuchi introduced
a natural Riemannian metric on $\mathcal{H}(X,\omega)$\emph{ \cite{mab-1},}
where the squared norm of a tangent vector $v\in C^{\infty}(X)$ at
$u$ is defined by 
\begin{equation}
g_{|u}(v,v):=\int_{X}v^{2}\omega_{u}^{n}\label{eq:mab metric intro}
\end{equation}
The main case of geometric interest is when the cohomology class $[\omega]$
is integral, which equivalently means that it can be realized as the
first Chern class $c_{1}(L)$ of an ample line bundle $L$ over the
projective algebraic manifold $X.$ Then the space $\mathcal{H}(X,\omega)$
may be identified with the space $\mathcal{H}(L)$ of all positively
curved metrics $\phi$ on the line bundle $L$ and as pointed by Donaldson
\cite{do00} the space $\mathcal{H}(L)$ may then be interpreted as
the symmetric space dual of the group $\mbox{Ham\ensuremath{(X,\omega)}}$
of Hamiltonian diffeomorphisms of $(X,\omega).$ Under this (formal)
correspondence the geodesics in $\mathcal{H}(X,\omega)$ correspond
to one-parameter subgroups in the (formal) complexification of $\mbox{Ham\ensuremath{(X,\omega)}}$
and this motivated Donaldson's conjecture concerning the existence
of geodesics in $\mathcal{H}(X,\omega),$ connecting any two given
elements. 

However, Donaldson's existence problem has turned out to be quite
subtle. In fact, according to the recent counter-examples in \cite{l-v,dar-l}
the existence of bona fide geodesic segments fails in general. On
the other hand, there always exists a (unique) \emph{weak }geodesic
$u_{t}$ connecting given points $u_{0}$ and $u_{1}$ in $\mathcal{H}(X,\omega)$
defined as follows. First recall that, by an important observation
of Semmes \cite{se} and Donaldson \cite{do00}, after a complexification
of the variable $t,$ the geodesic equation for $u_{t}$ on $X\times[0,1]$
may be written as the following complex Monge-Ampère equation on a
domain $M:=X\times D$ in $X\times\C$ for the function $U(x,t):=u_{t}(x):$
\begin{equation}
(\pi^{*}\omega+dd^{c}U)^{n\text{+1}}=0.\label{eq:ma eq for geod intro}
\end{equation}
As shown by Chen \cite{c0}, with complements by Blocki \cite{bl},
for any smoothly bounded domain $D$ in $\C$ the corresponding boundary
value problem on $M$ admits a unique solution $U$ such that $\pi^{*}\omega+dd^{c}U$
is a positive current with coefficients in $L^{\infty},$ satisfying
the equation \ref{eq:ma eq for geod intro} almost everywhere. In
particular, when $D$ is an annulus in $\C$ this construction gives
rise to the notion of a weak geodesic curve $u_{t}$ in the space
of all functions $u$ such that $\omega_{u}$ is a positive current
with coefficients in $L^{\infty}$ (the latter regularity equivalently
means that the Laplacian of $u$ is in $L^{\infty}$). In particular,
by standard linear elliptic estimates, $U$ is ``almost $C^{1,1}$''
in the sense that $U$ is in the Hölder class $C^{1,\alpha}$ for
any $\alpha<1.$ As shown by Blocki \cite{bl}, in the case when $X$
admits a Kähler metric with non-negative holomorphic bisectional curvature
Chen's regularity result can be improved to give that $U$ is $C^{1,1}-$smooth.
However, the assumption on $X$ appearing in Blocki's result is very
strong and essentially implies that $X$ is a homogenuous manifold.
In this short note we point out that, in the case when the given Kähler
class $[\omega]$ is an integral the function $u_{t}$ on $X$ is
in general, for any fixed $t,$ in $C^{1,1}(X),$ i.e. its first derivatives
are Lipschitz continuous. More precisely, the real Hessian of $u_{t}$
has bounded coefficients with a a bound which is independent of $t:$ 
\begin{thm}
\label{thm:reg of geod intro}For any integral Kähler class $[\omega]$
the weak geodesic $u_{t}$ connecting any two points $u_{0}$ and
$u_{1}$ in the space $\mathcal{H}(X,\omega)$ of $\omega-$Kähler
potentials has the property that, for any fixed $t,$ the function
$u_{t}$ is in $C^{1,1}(X).$ More precisely, the upper bound on the
sup norm on $X$ of the real Hessian of $u_{t}$ only depending on
an upper bound of sup norms of the real Hessians of $u_{0}$ and $u_{1}.$ 
\end{thm}
This regularity result should be compared with recent results of Darvas-Lempert
\cite{dar-l} showing that the solution $U(x,t):=u_{t}(x)$ is not,
in general, $C^{2}-$smooth up to the boundary of $M$ in (more precisely
$dd^{c}U$ is not represented by a continuous form). However, the
argument in \cite{dar-l}, which is inspired by a similar argument
in the case of $M=D$ for a pseudoconvex domain $D$ in $\C^{2}$
due to Bedford-Fornaess \cite{b-f}, does not seem to exclude the
possibility that $U$ be $C^{2}-$smooth in the\emph{ interior} of
$M.$ Anyway, the latter scenario appears to be highly unlikely in
view of the explicit counter-example of Gamelin-Sibony \cite{g-s}
to interior $C^{2}-$regularity for the case when $D$ is the unit-ball
in $\C^{2}.$ Note also that, since the bounds on the real Hessian
of $u_{t}$ are controlled by the Hessians of $u_{0}$ and $u_{1}$
the previous theorem shows that $PSH(X,\omega)\cap C^{1,1}(X)$ is
closed with respect to weak geodesics. By the very recent work of
Darvas \cite{dar} and Guedj \cite{gu} this the latter property equivalently
means that $PSH(X,\omega)\cap C^{1,1}(X)$ defines a geodesic subspace
of the metric completion of the space $\mathcal{H}$ equipped with
the Mabuchi metric.

The starting point of the proof of Theorem \ref{thm:reg of geod intro}
is the well-known Perron type envelope representation of the solution
to the Dirichlet problem for the complex Monge-Ampère operator. The
proof, which is inspired by Bedford-Taylor's approach in their seminal
paper \cite{b-t1}, proceeds by a straight-forward generalization
of the technique used in \cite{ber2} to establish the corresponding
regularity result for certain envelopes of positively curved metrics
in a line bundle $L\rightarrow X$ (which can be viewed as solutions
to a free boundary value problem for the complex Monge-Ampère equation
on $X).$ In fact, the situation here is considerably simpler than
the one in \cite{ber2} which covers the case when the line bundle
$L$ is merely big (the $C^{1,1}-$regularity then holds on the ample
locus of $L$ in $X)$ and one of the motivations for the present
note is to highlight the simplicity of the approach in \cite{ber2}
in the present situation (see also \cite{r-w} for other generalizations
of \cite{ber2}). But it should be stressed that, just as in \cite{ber2},
the results can be generalized to more general line bundles. For example,
by passing to a smooth resolution, Theorem \ref{thm:reg of geod intro}
be generalized to show that the weak geodesic connecting any two smooth
metrics with non-negative curvature current on an ample line bundle
$L$ over a singular compact normal complex variety $X$ is $C^{1,1}-$smooth
on the regular locus of $X$ (for a fixed ``time'').

As it turns out one can formulate a general result (Theorem \ref{thm:reg text}
below) which contains both Theorem \ref{thm:reg of geod intro} and
the corresponding regularity result in \cite{ber2}. In particular,
the latter result  covers the case when the domain $D$ is the unit-disc
(or more generally, the unit-ball in $\C^{n},$ where the following
more precise regularity result holds:
\begin{thm}
\label{thm:ref for ma over disc intro}For any integral Kähler class
$[\omega]$ on a compact complex manifold $X$ the solution $U$ to
the Dirichlet problem for the complex Monge-Ampère equation \ref{eq:ma eq for geod intro}
with $C^{2}-$boundary data, $\omega-$psh along the slices $\{t\}\times X,$
is $C^{1,1}-$smooth in the interior of $X\times D,$ if $D$ is the
unit-disc in $\C.$ 
\end{thm}
As pointed out by Donaldson \cite{do00} the boundary value problem
appearing in the previous theorem can be viewed as an infinite dimensional
analog of a standard boundary value problem for holomorphic discs
in the complexification of a compact Lie group $G$ or more precisely
the classical factorization theorem for loops in $G$ (recall that
the role of $G$ in the present infinite dimensional setting is played
by the group $\mbox{Ham\ensuremath{(X,\omega)}}$ of Hamiltonian diffeomorphisms).
As shown by Donaldson \cite{do1} the solution $U$ is in general
not smooth and Donaldson raised the problem of studying the singularities
of Chen's weak solution; the paper can thus be seen as one step in
this direction. 

One potentially useful consequence of the regularity results in Theorems
\ref{thm:reg of geod intro}, \ref{thm:ref for ma over disc intro}
is that, for a fixed ``time'' $t$ the differential of $u_{t}$
(which geometrically represents the connection one form of the corresponding
metric on the line bundle $L$) is Lipschitz continuous and in particular
differentiable on $X-E,$ where the exceptional set $E$ is a null
set for the Lebesgue measure. For example, it then follows from the
results in \cite{ber2} that the corresponding scaled Bergman kernel
$B_{k}(x,x)/k^{n},$ attached to high tensor powers $L^{\otimes k},$
converges when $k\rightarrow\infty$ point-wise on $X-E$ to the density
of $\omega_{u_{t}}^{n}.$ By a circle of ideas going back to Yau such
Bergman kernels can be used to approximate differential geometric
objects in Kähler geometry. Accordingly, the precise $C^{1,1}-$regularity
established in the present paper will hopefully find applications
in Kähler geometry in the future. In fact, one of the initial motivations
for writing the present note came from a very recent joint work with
Bo Berndtsson \cite{b-b} where Bergman kernel asymptotics are used
to establish the convexity of Mabuchi's K-energy along weak geodesics
and where the precise $C^{1,1}-$regularity was needed at an early
stage of the work. Eventually it turned that Chen's regularity, or
more precisely the fact that $u_{t}$ has a bounded Laplacian, is
sufficient to get the point-wise convergence of $B_{k}/k^{n}$ for
some \emph{subsequence} away from some (non-explicit) null set $E$
(see Theorem 2.1 in \cite{b-b}) which is enough to run the approximation
argument. But with a bit of imagination one could envisage future
situations where the more precise $C^{1,1}-$regularity would be needed.

Let us finally point out that in a very recent article Darvas and
Rubinstein \cite{d-r} consider psh-envelopes of functions of the
form $f=\min\{f_{1},f_{2},...,f_{m}\}.$ Such envelopes appear in
the Legendre transform type formula for weak geodesics introduced
in \cite{d-r} which has remarkable applications to the study of the
completion of the Mabuchi metric space \cite{dar}. The same technique
from \cite{ber2} we describe here implies $C^{1,1}$ regularity of
such envelopes in the case the Kähler class is integral (see the first
point in Section \ref{sub:Further-remarks}). In \cite{d-r} the authors
give a different proof of this result (still using \cite{ber2}) and
also prove a Laplacian bound in the case of a general Kähler class.

\subsection*{Acknowledgments}

I am grateful to Jean-Pierre Demailly for many stimulating and illuminating
discussions on the topic of the present note. Also thanks to Slawomir
Dinew for comments and Yanir Rubinstein for discussions related to
\cite{d-r}. This work has been supported by grants from the Swedish
and European Research Councils and the Wallenberg Foundation.

\section{$C^{1,1}-$regularity of solutions to complex Monge-Ampère equations
over products}

\subsection{\label{sub:Notation-and-preliminaries}Notation: quasi-psh functions
vs metrics on line bundles}

Here we will briefly recall the notion for (quasi-) psh functions
and metrics on line bundles that we will use. Let $(X,\omega_{0})$
be a compact complex manifold of dimension $n$ equipped with a fixed
Kähler form $\omega_{0},$ i.e. a smooth real positive closed $(1,1)-$form
on $X.$ Denote by $PSH(X,\omega_{0})$ be the space of all $\omega_{0}-$psh
functions $u$ on $X,$ i.e. $u\in L^{1}(X)$ and $u$ is (strongly)
upper-semicontinuos (usc) and 
\[
\omega_{u}:=\omega_{0}+\frac{i}{2\pi}\partial\bar{\partial}u:=\omega_{0}+dd^{c}u\geq0,
\]
 holds in the sense of currents. We will write $\mathcal{H}(X,\omega_{0})$
for the interior of $PSH(X,\omega_{0})\cap\mathcal{C}^{\infty}(X),$
i.e. the space of all Kähler potentials (w.r.t $\omega_{0}).$ In
the\emph{ integral case}, i.e. when $[\omega]=c_{1}(L)$ for a holomorphic
line bundle $L\rightarrow X,$ the space $PSH(X,\omega_{0})$ may
be identified with the space $\mathcal{H}_{L}$ of (singular) Hermitian
metrics on $L$ with positive curvature current. We will use additive
notion for metrics on $L,$ i.e. we identify an Hermitian metric $\left\Vert \cdot\right\Vert $
on $L$ with its ``weight'' $\phi.$ Given a covering $(U_{i},s_{i})$
of $X$ with local trivializing sections $s_{i}$ of $L_{|U_{i}}$
the object $\phi$ is defined by the collection of open functions
$\phi_{|U_{i}}$ defined by 
\[
\left\Vert s_{i}\right\Vert ^{2}=e^{-\phi_{|U_{i}}}.
\]
The (normalized) curvature $\omega$ of the metric $\left\Vert \cdot\right\Vert $
is the globally well-defined $(1,1)-$current defined by the following
local expression: 
\[
\omega=dd^{c}\phi_{|U_{i}}.
\]
The identification between $\mathcal{H}_{L}$ and $PSH(X,\omega_{0})$
referred to above is obtained by fixing $\phi_{0}$ and identifying
$\phi$ with the function $u:=\phi-\phi_{0},$ so that $dd^{c}\phi=\omega_{u}.$

\subsection{\label{sub:The-regularity-of}The $C^{1,1}-$regularity of weak geodesics}

Let $(X,\omega)$ be a compact Kähler manifold and $D$ a domain in
$\C^{n}.$ Set $M:=X\times D$ and denote by $\pi$ the natural projection
from $M$ to $X.$ Given a continuous function $f$ on $\partial M(=X\times\partial D)$
we define the following point-wise Perron type upper envelope on the
interior of $M:$

\begin{equation}
U:=P(f):=\sup\{V:\,\,\, V\in\mathcal{F}\},\label{eq:envelope def in section reg}
\end{equation}
 where $\mathcal{F}$ denotes the set of all $V\in PSH(M,\pi^{*}\omega)$
such that $V_{|\partial M}\leq f$ on the boundary $\partial M$ (in
a point-wise limiting sense). In the case when $D$ is a smoothly
bounded pseudoconvex domain and $f$ is $\omega-$psh in the ``$X-$directions'',
i.e. $f(\cdot,t)\in PSH(X,\omega)$ it was shown in \cite{b-d} that
$P(f)$ is continuous up to the boundary of $M$ and $U$ then coincides
with the unique solution of the Dirichlet problem for the corresponding
complex Monge-Ampère operator with boundary data $f,$ in the weak
sense of pluripotential theory \cite{b-t1}. Here we will establish
the following higher order regularity result for the envelope $P(f):$
\begin{thm}
\label{thm:reg text}Let $(X,\omega_{0})$ be an $n-$dimensional
integral compact Kähler manifold manifold and $D$ a bounded domain
in $\C^{m}$ and set $M:=X\times D.$ Then, given $f$ a function
on $\partial M$ such that $f(\cdot,\tau)$ is in $C^{1,1}(X),$ with
a uniform bound on the corresponding real Hessians, the function $u_{\tau}:=P(f)_{|X\times\{\tau\}}$
is in $C^{1,1}(X)$ and satisfies 
\[
\sup_{X}|\nabla^{2}u_{\tau}|_{\omega_{0}}\leq C,
\]
 where $|\nabla^{2}u_{\tau}|_{\omega_{0}}$ denotes the point-wise
norm of the real Hessian matrix of the function $u_{\tau}$ on $X$
defined with respect to the Kähler metric $\omega_{0}.$ Moreover,
the constant $C$ only depends on an upper bound on the sup norm of
the real Hessians of $f_{\tau}$ for $\tau\in\partial D.$ In the
case when $D$ is the unit-ball the function $U(x,\tau)$ is in $C_{loc}^{1,1}$
in the interior of $M.$
\end{thm}

\subsubsection{Proof of Theorem \ref{thm:reg text}}

In the course of the proof of the theorem we will identify an $\pi^{*}\omega-$psh
function $U$ on $M$ with a positively curved metric $\Phi$ on the
line bundle $\pi^{*}L\rightarrow M.$ The case when $D$ is a point
is the content of Theorem 1.1 in \cite{ber2} and as will be next
explained the general case can be proved in completely analogous manner.
First recall that the argument in \cite{ber2} is modelled on Bedford-Taylor's
proof of the case when $X$ is a point and $D$ is the unit-ball \cite{b-t1}
(see also Demailly's simplifications \cite{de}). The latter proof
uses that $B$ is a homogenuous domain. In order to explain the idea
of the proof of Theorem \ref{thm:reg text} first consider the case
when $(X,L)$ is \emph{homogenuous}, i.e. the group $\mbox{Aut}\mbox{ \ensuremath{(X,L)}}$
of all biholomorphic automorphisms of $X$ lifting to $L$ acts transitively
on $X.$ In particular, there exists a family $F_{\lambda}$ in $\mbox{Aut}\mbox{ \ensuremath{(X,L)}}$
parametrized by $\lambda\in\C^{n}$ such that, for any fixed point
$x\in X,$ the map $\lambda\mapsto F_{\lambda}(x)$ is a biholomorphism
(onto its image) from a sufficiently small ball centered at the origin
in $\C^{n}.$ Given a metric $\phi$ on $L$ we set $\phi^{\lambda}:=F_{\lambda}^{*}\phi.$
Similarly, given a metric $\Phi(=\Phi(x,\tau))$ on $\pi^{*}L$ we
set 
\[
\Phi^{\lambda}:=(F_{\lambda}\times I)^{*}\Phi.
\]
 Since $F_{\lambda}$ is holomorphic the metric $\Phi^{\lambda}$
has positive curvature iff $\Phi$ has positive curvature. Now to
first prove a Lipschitz bound on $P\Phi_{f},$ where $\Phi_{f}$ is
the metric on $L\rightarrow\partial M$ corresponding to the given
boundary data $f,$ we take any candidate $\Psi$ for the sup defining
$P\Phi_{f}$ and note that, on $\partial M,$ i.e. for $\tau\in\partial D:$
\begin{equation}
\Psi^{\lambda}\leq\Phi_{f}^{\lambda}\leq\Phi_{f}+C_{1}|\lambda|,\label{eq:lip bound for cand}
\end{equation}
 where $C_{1}$ only depends on the Lipschitz bounds in the $"X-$direction''
of the given function $f$ on $X\times\partial D.$ But this means
that $\Psi^{\lambda}-C_{1}|\lambda|$ is also a candidate for sup
defining $P\Phi_{f}$ and hence $\Psi^{\lambda}-C_{1}|\lambda|\leq P\Phi_{f}$
on all of $X\times D.$ Finally, taking the sup over all candidates
$\Psi$ gives, on $X\times D,$ that 
\[
(P\Phi_{f})^{\lambda}\leq(P\Phi_{f})+C_{1}|\lambda|
\]
Since this holds for any $\lambda$ and in particular for $-\lambda$
this concludes the proof of the desired Lipschitz bound on $P\Phi_{f}.$
Next, to prove the bound on the real Hessian one first replaces $\Psi^{\lambda}$
in the previous argument with $\frac{1}{2}(\Psi^{\lambda}+\Psi^{-\lambda})$
and deduces, precisely as before, that 
\[
\frac{1}{2}\left((P\Phi_{f})^{\lambda}+(P\Phi_{f})^{-\lambda}\right)\leq(P\Phi_{f})+C_{2}|\lambda|^{2},
\]
 where now $C_{2}$ depends on the upper bound in the $"X-$direction''
of the real Hessian of the function $f$ on $X\times\partial D.$
The previous inequality implies an upper bound on the real Hessians
of the local regularizations $\Psi_{\epsilon}$ of $P\Phi_{f}$ defined
by local convolutions. Moreover, since $dd^{c}\Psi_{\epsilon}\geq0$
it follows from basic linear algebra that a lower bound on the real
Hessians also holds. Hence, letting $\epsilon\rightarrow0$ shows
that $P\Phi_{f}$ is in $C_{loc}^{1,1}$ in the $"X-$direction''
with a uniform upper bound on the real Hessians (compare \cite{b-t1,de}).

Of course, a general polarized manifold $(X,L)$ may not admit even
a single (non-trivial) holomorphic vector field. But as shown in \cite{ber2}
this problem can be circumvented by passing to the total space $Y$
of the dual line bundle $L^{*}\rightarrow X,$ which does admit an
abundance of holomorphic vector fields. The starting point is the
standard correspondence between positively curved metrics $\phi$
on $L$ and psh ``log-homogenuous'' functions $\chi$ on $Y$ induced
by the following formula:

\[
\chi(z,w)=\phi(z)+\log|w|^{2},
\]
 where $z$ denotes a vector of local coordinates on $X$ and $(z,w)$
denote the corresponding local coordinates on $Y$ induced by a local
trivialization of $L.$ Accordingly, the envelope $P\Phi_{f}$ on
$X$ corresponds to an envelope construction on $Y,$ defined w.r.t
the class of psh log-homogenous functions on $Y.$ Fixing a metric
$\phi_{0}$ on $L$ we denote by $K$ the compact set in $Y$ defined
by the corresponding unit-circle bundle. By homogenity any function
$\chi$ as above is uniquely determined by its restriction to $K.$
Now, for any fixed point $y_{0}$ in $K$ there exists an $(n+1)-$tuple
of global holomorphic vector fields $V_{i}$ on $Y$ defining a frame
in a neighborhood of $y_{0}:$ 
\begin{lem}
Given any point $y_{0}$ in the space $Y^{*}$ defined as the complement
of the zero-section in the total space of $L^{*}$ there exist holomorphic
vector fields $V_{1},...,V_{n+1}$ on $Y^{*}$ which are linearly
independent close to $y_{0}.$ \end{lem}
\begin{proof}
This follows from Lemma 3.7 in \cite{ber2}. For completeness and
since we do not need the explicit estimates furnished by Lemma 3.7
in \cite{ber2} we give a short direct proof here. Set $Z:=\P(L^{*}\oplus\C),$
viewed as the fiber-wise $\P^{1}-$compactification of $Y.$ Denote
by $\pi$ the natural projection from $Z$ to $X$ and by $\mathcal{O}(1)$
the relative (fiberwise) hyper plane line bundle on $Z.$ As is well-known,
for any sufficiently positive integer the line bundle $L_{m}:=(\pi^{*}L)\otimes\mathcal{O}(1)^{\otimes m}$
on $Z$ is ample and holomorphically trivial on $Y^{*}.$ As a consequence,
the rank $n+1-$ vector bundle $E:=TZ\otimes L_{m}^{\otimes k}$ is
globally generated for $k$ sufficiently large, i.e. any point $z_{0}$
in $Z$ there exists global holomorphic sections $S_{1},.,,,S_{n+1}$
spanning $E_{|z_{0}}.$ Since, $L_{m}$ is holomorphically trivial
on $Y^{*}\subset Z$ this concludes the proof.
\end{proof}
Now, integrating the (short-time) flow of the holomorphic vector field
$V(\lambda):=\sum\lambda_{i}V_{i}$ gives a family of holomorphic
maps $F_{\lambda}(y)$ defined for $y\in K$ and $\lambda$ in a sufficiently
small ball $B$ centered at the origin in $\C^{n+1}$ such that $\lambda\mapsto F_{\lambda}(y_{0})$
is a biholomorphism. However, the problem is that the corresponding
function $\chi^{\lambda}:=F_{\lambda}^{*}\chi$ is only defined in
a neighborhood of $K$ in $Y$ (and not log-homogenuous). But this
issue can be bypassed by replacing $\chi^{\lambda}$ with a new function
that we will denote by $T(\chi^{\lambda}),$ where $T(f),$ for $f$
a function on $K,$ is obtained by first taking the sup of $f$ over
the orbits of the standard $S^{1}-$ action on $Y$ to get an $S^{1}-$invariant
function $g:=\hat{f}$ and then replacing $g$ with its log-homogenuous
extension $\tilde{g},$ i.e. 
\[
T(f):=\widetilde{\left(\widehat{\chi^{\lambda}}\right)}.
\]
The following lemma follows from basic properties of plurisubharmonic
functions (see \cite{ber2} for a proof):
\begin{lem}
If $f$ is the restriction to the unit-circle bundle $K\subset Y$
of a psh function, then $T(f)$ is a psh log homogenuous function
on $Y$
\end{lem}
Now performing the previous constructions for any fixed $\tau\in D$
and identifying a candidate $\Psi$ with a function $\chi$ on $Y\times D,$
as above, gives 
\begin{equation}
\chi^{\lambda}(y_{0})\leq\widehat{\chi^{\lambda}}(y_{0})=\widetilde{\left(\widehat{\chi^{\lambda}}\right)}(y_{0}):=T(\chi^{\lambda})(y_{0}).\label{eq:ineq in proof of c11 reg}
\end{equation}
But, by construction, for $\tau\in\partial D$ we have $T(\chi^{\lambda})\leq T(\chi_{\Phi_{f}}^{\lambda})$
and since $f_{\tau}$ is assumed Lipschitz for $\tau\in\partial D$
we also have that 
\[
T(\chi_{\Phi_{f}}^{\lambda})\leq T(\chi_{\Phi_{f}})+C_{1}|\lambda|=\chi_{\Phi_{f}}+C_{1}|\lambda|.
\]
 But this means that $T(\chi^{\lambda})-C_{1}|\lambda|$ is a candidate
for the sup in question and hence bounded from above by $\chi_{P\Phi_{f}}$,
which combined with the inequality \ref{eq:ineq in proof of c11 reg}
gives 
\[
\chi^{\lambda}(y_{0})-C_{1}|\lambda|\leq\chi_{\Phi_{f}}(y_{0}).
\]
 Taking the sup over all candidates $\chi$ and replacing $\lambda$
with $-\lambda$ hence gives the desired Lipschitz bound on $P\Phi_{f}$
at the given point $y_{0}$ and hence, by compactness, for any point
in $K.$ The estimate on the Hessian then proceeds precisely as above. 

Finally, in the case when $B$ is the unit-ball one can exploit that
$B$ is homogenuous (under the action of the Möbius group), replacing
the holomorphic maps $(x,\tau)\mapsto(F_{\lambda}(x),\tau)$ used
above with $(x,\tau)\mapsto(F_{\lambda}(x),G_{a}(\tau)),$ where $G_{a}$
is a suitable family of Möbius transformations (the case when $X$
is point is precisely the original situation in \cite{b-t1}). Then
the proof proceeds precisely as before.

\subsection{\label{sub:Further-remarks}Further remarks}
\begin{itemize}
\item The proof of the previous theorem also applies in the more general
situation where $f$ may be written as $f=\inf_{\alpha\in A}f_{\alpha}$
for a given family of functions $f_{\alpha},$ as long as the Hessians
of $f_{\alpha}(\tau,\cdot)$ are uniformly bounded on $X$ (by a constant
$C$ independent of $\tau$ and $\alpha))$ and similarly for the
Lipschitz bound. Indeed, then equation \ref{eq:lip bound for cand}
holds with $f$ replaced by $f_{\alpha}$ for any $\alpha\in A$ with
the same constant $C.$ For $D$ equal to a point this result has
been obtained in \cite{d-r} using a different proof.
\item As shown in \cite{b-d} (using a different pluripotential method),
in the case of a general, possibly non-integral, Kähler class $[\omega]$
a bounded Laplacian in the $X-$directions of the boundary data $f$
results in a bounded Laplacian of the corresponding envelope. In the
case of geodesics this result has also recently been obtained in \cite{he}
by refining Chen's proof.
\item By the proof of the previous theorem, the Lipschitz norm $\left\Vert u_{t}\right\Vert _{C^{0,1}(X)}$
of a weak geodesic $u_{t}$ only depends on an upper bound on the
Lipschitz norms of $u_{0}$ and $u_{1}.$ Since the Lipschitz norm
in the $t-$variable is controlled by the $C^{0}-$norm of $u_{0}-u_{1}$
\cite{bern1} it follows that the Lipschitz norm $\left\Vert U\right\Vert _{C^{0,1}(X\times A)}$
of the corresponding solution $U$ on $X\times A$ is controlled by
the Lipschitz norms of $u_{0}$ and $u_{1}$ and the $C^{0}-$norm
of $u_{0}-u_{1}.$ For a general Kähler class this result also follows
from Blocki's gradient estimate \cite{bl0,bl}. \end{itemize}

\end{document}